\documentclass[12pt,a4]{amsart}
\usepackage[a4paper, left=28mm, right=28mm, top=28mm, bottom=28mm]{geometry}
\usepackage[utf8]{inputenc}
\usepackage[all]{xy}
\usepackage{amsmath}
\usepackage{amssymb}
\usepackage{amsthm}
\usepackage{amsfonts}
\usepackage{mathrsfs}
\usepackage{mathtools}
\usepackage{makecell}
\usepackage{hyperref}
\usepackage{breakurl}

\usepackage[OT2,T1]{fontenc}
\DeclareSymbolFont{cyrletters}{OT2}{wncyr}{m}{n}
\DeclareMathSymbol{\Sha}{\mathalpha}{cyrletters}{"58}

\theoremstyle{definition}
\newtheorem{thm}{Theorem}[section]
\newtheorem{lem}[thm]{Lemma}
\newtheorem{cor}[thm]{Corollary}
\newtheorem{prop}[thm]{Proposition}

\theoremstyle{definition}

\newtheorem{rem}[thm]{Remark}

\def\F{{\mathbb F}}

\def\Q{{\mathbb Q}}

\def\Z{{\mathbb Z}}

\def\Gal{\mathop{\mathrm{Gal}}}
\def\GL{{\mathop{\mathrm{GL}}}}

\def\SL{{\mathop{\mathrm{SL}}}}

\numberwithin{equation}{section}

\title[modularity of elliptic curves]{modularity of elliptic curves over cyclotomic $\Z_p$-extensions of real quadratic fields}
\author{Sho Yoshikawa }
\date{June 25th, 2022}

\address{
Gakushuin University, Department of Mathematics, 1-5-1, Mejiro, Toshima-ku, Tokyo, 171--8588, Japan}
\email{yoshikawa@math.gakushuin.ac.jp}

\begin{document}

\maketitle
\begin{abstract}
We prove that all elliptic curves defined over the cyclotomic $\Z_p$-extension of a real quadratic field are modular under the assumption that the algebraic part of the central value of a twisted $L$-function is a $p$-adic unit.
Our result is a generalization of a result of X.\ Zhang, which is a real quadratic analogue of a result of Thorne.
\end{abstract}

\section{Introduction}
Let $E$ be an ellptic curve over a totally real  field $F$. 
It is conjectured that $E$ is \textit{modular}, in the sense that there exists a Hilbert modular cuspidal eigenform $f$ over $F$ of parallel weight 2 which satisfies $L(E,s)=L(f,s)$.
This conjecture is solved for all totally real fields $F$ with $[F:\Q]\leq 3$ (\cite{Wi}, \cite{TW}, \cite{BCDT}, \cite{FLS}, \cite{DNS}) and for all but finitely many totally real fields $F$ with $[F:\Q]\leq 5$ (\cite{Bo}, \cite{IIY}). The author solved it when $F/\Q$ is abelian and $3,5,$ and $7$ are unramified in $F$.
Also, Thorne solved it when $F$ is cointained in the cyclotomic $\Z_p$-extension $\Q_\infty$ of $\Q$ for a prime number $p$ (\cite[Theorem 1]{Th:cyclo}).
However, it is still open in general.

In this paper, we consider a real quadratic analogue of the result of Thorne. 
Let $X_0(15)$ (resp.\ $X_0(21)$) be the modular curve associated to the congruence subgroup $\Gamma_0(15)$ (resp.\ $\Gamma_0(21)$). Note that $X_0(15)$ (resp.\ $X_0(21)$) is the  elliptic curve over $\Q$ with Cremona label 15A1 (resp.\ 21A1); see Section \ref{reduction}.
Let $d$ be a square free  integer greater than $1$.
For $X=X_0(15)$ or $X_0(21)$, let $X_d$ be the quadratic twist of $X$ by the Dirichlet  character associated with $\Q(\sqrt{d})/\Q$.
The ratio  $L(X_d,1)/\Omega_{X_d}$ is a rational number,
where $\Omega_{X_d}$ is the real period of $X_d$; see \cite[\S 2.8]{Cr} for example.  
Let $p$ be a prime number,
and $\Q(\sqrt{d})_\infty$ be the cyclotomic $\Z_p$-extension of $\Q(\sqrt{d})$.
Let $F$ be a totally real field contained in $\Q( \sqrt{d})_\infty$.

The main results of this paper are stated as follows:
\begin{thm}
\label{main5}
Assume that the following conditions hold:
\begin{itemize}
\item $p\neq 2, 3$, or $5$.
\item $d \neq 5$.
\item $(d,3p)=1$ or $(d,5p)=1$.
\item $L(X_0(15)_d,1)/\Omega_{X_0(15)_d}$ is a $p$-adic unit.
\end{itemize}
Then all elliptic curves over $F$ are modular.
\end{thm}
\begin{thm}
\label{main7}
Assume that the following conditions hold:
\begin{itemize}
\item $p\neq 2, 3$, or $7$.
\item $7$ does not divide $d$.
\item $(d,3p)=1$ or $(d,7p)=1$.
\item $L(X_0(21)_d,1)/\Omega_{X_0(21)_d}$ is a $p$-adic unit.
\end{itemize}
Then all elliptic curves over $F$ are modular.
\end{thm}

\begin{rem}
\label{equiv}
(1) The assumption of Theorem \ref{main5} can be rephrased as follows:
\begin{itemize}
    \item $d\neq 5$. 
    \item $(d,3)=1$ or $(d,5)=1$.
    \item $p$ does not divide $2\cdot 3\cdot 5\cdot d$.
    \item $L(X_0(15)_d,1)/\Omega_{X_0(15)_d}$ is a $p$-adic unit.
\end{itemize}
(2) Similarly, the assumption of Theorem \ref{main7} are equivalent to the following:
\begin{itemize}
    \item $7$ does not divide $d$. 
    \item $p$ does not divide $2\cdot 3\cdot 5\cdot d$.
    \item $L(X_0(21)_d,1)/\Omega_{X_0(21)_d}$ is a $p$-adic unit.
\end{itemize}
\end{rem}

\begin{rem}
In \cite{Y:ab}, the author proves the following: If $F$ is a totally real field such that $F/\Q$ is abelian and $3,5$, and $7$ are unramified in $F$, then all elliptic curves over $F$ is modular. (In fact, the assumption  that $3,5$, and $7$ are unramified in $F$ can be replaced by the assumption that $3$ and $7$ are unramified in $F$ and $\sqrt{5}\notin F$. This is because the proof in \cite{Y:ab} just applies Thorne's modularity result  \cite[Theorem 7.6]{Th:irred} when the prime 5 is involved.)
Therefore, if $(pd, 105)=1$, Theorem \ref{main5} and Theorem \ref{main7}   follow from \cite[Theorem 1.2]{Y:ab}.
\end{rem}
\begin{rem}
Recently, X.\ Zhang studied the case where $F$ is contained in the cyclotomic $\Z_p$-extension of certain real quadratic fields (\cite{Z}).
We describe three differences between our results and the results of X.\ Zhang.
\begin{itemize}
    \item First, it is assumed in \cite[Theorem 1.1]{Z} that $p$ splits in $\Q(\sqrt{d})$, but Theorem \ref{main5} and Theorem \ref{main7} do not need such a splitting condition on $p$.
\item Second, \cite{Z} assumes that  $d\neq 5$. We need the same condition in Theorem \ref{main5},  
but Theorem \ref{main7} does not require such an assumption.
\item Finally, \cite{Z} makes an assumption on the $p$-part of the Mordell--Weil groups and the Tate-Shafarevich groups of certain elliptic curves over certain real quadratic fields.
In Theorem \ref{main5} and Theorem \ref{main7}, we instead assume that the algebraic part of a central  $L$-value is a $p$-adic unit, which is relatively easy to check.
\end{itemize}
\end{rem}

In our proof of Theorem \ref{main5} and Theorem \ref{main7}, we follow Thorne's strategy in \cite{Th:cyclo}.
Recall that he first uses his modularity results \cite[Theorem 7.6]{Th:irred} for elliptic curves whose mod 5 representations are irreducible, and then he reduces the problem to the study of $\Q_\infty-$rational points on the modular curves $X(s3,b5)$ and $X(b3,b5)=X_0(15)$; see Section \ref{mainproof} for the notation. 
Then he uses Iwasawa theory of elliptic curves.
In our proof of Theorem \ref{main5}, we reduce the problem to the study of $\Q_\infty$-rational points on a quadratic twist of $X_0(15)$, and we apply the method of Thorne to the quadratic twist of $X_0(15)$. 
In our proof of Theorem \ref{main7},  we instead use the modularity results by the author \cite[Theorem 1.5]{Y:ab} for elliptic curves whose mod $7$ representations are irreducible, and then we reduce the problem to the study of $\Q_\infty$-rational points on  quadratic twists of  $X_0(21)$. 
We then apply the method of Thorne to the quadratic twists of $X_0(21)$.
The details are given in Section \ref{reduction} and Section \ref{mainproof}.



\section{Results from Iwasawa theory of elliptic curves}
In this section, we collect some facts about Iwasawa theory of elliptic curves for the readers' convenience. The first two results concern elliptic curves with good ordinary reduction at a given prime.

The following theorem follows from the results of Kato.

\begin{thm}[Kato]
\label{Theorem:Kato}
Let $p$ be an odd prime and $E$ be an elliptic curve over $\Q$ with conductor $N_E$.
Assume that the following conditions are satisfied:
\begin{itemize}
\item $E$ has good ordinary reduction at $p$.
\item The mod $p$ representation $\overline{\rho}_{E,p}\colon \Gal(\overline{\Q}/\Q)\to \GL_2(\F_p)$ is surjective.
\item There exists a prime $q\neq p$ such that $q||N_E$ and  $\overline{\rho}_{E,p}$ is ramified at $q$.
\item $L(E, 1)\neq 0$.
\end{itemize}
Then the $p$-part of the Tate--Shafarevich group $\Sha(E/\Q)_p$
is finite, and
$L(E,1)/\Omega_E$ and $\#(\Sha(E/\Q)_p) \prod_{l|N_E} c_l(E)$
have the same $p$-adic vauation, where $c_l(E)$ denotes the Tamagawa factor of $E$ at $l$.
\end{thm}

\begin{proof}
See \cite[Theorem 17.4]{Ka}. See also  \cite[Theorem 3.35]{SU} for a convenient reference.
\end{proof}

\begin{rem}
Just after \cite[Theorem 3.35]{SU}, it is remarked that the hypotheses of Theorem \ref{Theorem:Kato} are satisfied if $E$ is semistable and $p\geq 11$ is a prime of good ordinary reduction.
\end{rem}
We also need the following theorem, which is specialized to the case where the base field is $\Q$.

\begin{thm}[Greenberg]
\label{Theorem:Greenberg}
Let $E$ be an elliptic curve over $\Q$ with conductor $N_E$.
Assume that $E$ has good ordinary reduction at a prime $p$.
Assume also that $\mathrm{Sel}(E/\Q)_p$ is finite.
Let $f_E(T)$ be a generator of the characteristic ideal of $\mathrm{Hom}(\mathrm{Sel}(E/\Q_\infty)_p, \Q_p/\Z_p)$. 
Then $E(\Q)$ is finite, and $f_E(0)$ and the value
\[ \left(\prod_{l  l \mid N_E}c_{l}(E)\right)(\#\widetilde{E}_p(\F_p))^2\cdot  \# \mathrm{Sel}(E/\Q)_p/(\#E(\Q))^2\]
have the same $p$-adic valuation.
Here, $\widetilde{E}_p$ denotes the mod $p$ reduction of $E$.
\end{thm}

\begin{proof}
See \cite[Theorem 4.1]{Gr}.
\end{proof}

Combining the results of Kato and Greenberg,
we have the following corollary.

\begin{cor}
\label{ordfin}
Let $p$ be an odd prime and $E$ be an elliptic curve over $\Q$ with conductor $N_E$.
Assume that the following conditions are satisfied:
\begin{enumerate}
\item $E$ has good ordinary reduction at $p$.
\item The mod $p$ representation $\overline{\rho}_{E,p}\colon \Gal(\overline{\Q}/\Q)\to \GL_2(\F_p)$ is surjective.
\item There exists a prime $q\neq p$ such that $q||N_E$ and and $\overline{\rho}_{E,p}$ is ramified at $q$.
\item $L(E, 1)/\Omega_E$ is a $p$-adic unit.
\item $p$ does not divide $\#\widetilde{E}_p(\F_p)$.
\item $p$ does not divide $\#E(\Q)$.
\end{enumerate}
Then $E(\Q_\infty)$ is a finite group.
\end{cor}

\begin{proof}
Note that by \cite[Corollary 14.3]{Ka} the condition (4) implies the finiteness of $E(\Q)$, and hence the condition (6) makes sense.
By Theorem \ref{Theorem:Kato},
we see that $\#(\Sha(E/\Q)_p)$ and $\prod_{l|N_E} c_l(E)$ are $p$-adic units.
From the short exact sequence
\[0\to E(\Q)\otimes_\Z \Q_p/\Z_p\to \mathrm{Sel}(E/\Q)_p\to \Sha(E/\Q)_p\to 0. \]
we see that
$\# \mathrm{Sel}(E/\Q)_p$ is a $p$-adic unit.
By Theorem \ref{Theorem:Greenberg},
we see that $f_E(0)$ is a $p$-adic unit.
Therefore the $\mu$-invariant and the $\lambda$-invariant of
$\mathrm{Hom}(\mathrm{Sel}(E/\Q_\infty)_p, \Q_p/\Z_p)$
are both zero.
We conclude that $E(\Q_\infty)$ is a finite group.
\end{proof}


Next we need the following result of Kurihara, which deals with elliptic curves with good supersingular reduction at a given prime.
\begin{thm}[Kurihara]
\label{Ku} Let $p$ be an odd prime number and $E$ be an elliptic curve over $\Q$ with conductor $N_E$.
Assume that the following conditions are satisfied:
\begin{enumerate}
\item $E$ has good supersingular reduction at $p$.
\item The mod $p$ representation $\overline{\rho}_{E,p}\colon\Gal(\overline{\Q}/\Q) \to \GL_2(\F_p)$ is surjective.  
\item $L(E,1)/\Omega_E$ is a $p$-adic unit.
\item $p$ does not divide $\prod_{l \mid N_E}c_l(E).$
\end{enumerate}
Then $E(\Q_\infty)$ is a finite group.
where $\Q_n$ denotes the unique subextension of $\Q_\infty/\Q$ with $[\Q_n:\Q]=p^n$.
\end{thm}

\begin{proof}
We see from \cite[Theorem 0.1 (2)]{Ku} that $E(\Q_\infty)$ is a torsion group, and it is in fact finite by \cite{Im}. 
\end{proof}


\section{The reduction step}
\label{reduction}
In this section, we reduce the proof of Theorem \ref{main5} and Theorem \ref{main7} to the study of $F_\infty$-rational points on $X_0(15)$ and $X_0(21)$, where $F$ is a real quadratic field.

Let $q=5$ or $7$.
We consider two modular curves $X(s3, bq)$ and $X(b3,bq)$; for the notation on modular curves, see  \cite[\S 2.2]{FLS}.
Roughly speaking, $X(s3,bq)$ is the modular curve which classifies elliptic curves $E$ such that the mod 3 representations $\overline{\rho}_{E,3}$ are irreducible and that the mod $q$ representations $\overline{\rho}_{E,q}$ are reducible.
Also, $X(b3,bq)$ is the modular curve  which classifies elliptic curves $E$ such that both $\overline{\rho}_{E,3}$ and $\overline{\rho}_{E,q}$ are reducible.
Note that $X(b3,bq)$ is isomorphic to the modular curve $X_0(3q)$ corresponding to the congruence subgroup $\Gamma_0(3q)$.
\begin{lem}
\label{mc}
The following hold.
\begin{enumerate}
    \item The modular curve $X_0(15)$ (resp.\ $X(s3,b5$)) is isomorphic over $\Q$ to the elliptic curve with Cremona label 15A1 (resp.\ 15A3).
    \item There is an isogeny $X(s3,b5)\to X_0(15)$ of degree 2 over $\Q$.
    \item The modular curve $X_0(21)$ is isomorphic over $\Q$ to the elliptic curve with Cremona label 21A1.
    \item There is a morphism $X(s3,b7)\to X_0(21)$ of degree 2 over $\Q$.
    \item The $j$-invariant of $X_0(15)$ (resp. $X_0(21)$) is $3^{-4}\cdot 5^{-4}\cdot 13^3\cdot 37^3$ (resp. $3^{-4}\cdot 7^{-2}\cdot 193^3$).
\end{enumerate}
\end{lem}
\begin{proof}
 For (1), see \cite[Lemma 5.6, Lemma 5.7]{FLS}. The part (2) is checked at LMFDB \cite{L}.
 For (3) and (4), see the proof of \cite[Lemma 1.1]{FLS}:
 Note that the modular curve $X(s3,b7)$ is isomorphic to the Atkin--Lehner quotient $X_0(63)/\langle w_9 \rangle$ of genus 3, and its quotient $X_0(63)/\langle w_9,w_7\rangle$ is the elliptic curve with Cremona label 21A1, which is isomorphic to $X_0(21)$.
 The part (4) can be checked at LMFDB.
\end{proof}
\begin{lem}
\label{points}
Let $X, X'$ be two smooth projective curves over a field $F$ with a finite morphism $f\colon X'\to X$ of degree $2$. 
Let $F'/F$ be a finite extension of fields of odd degree.
If $X(F')=X(F)$, then we have $X'(F')=X'(F)$.
\end{lem}
\begin{proof}
Since the inclusion $X'(F)\subset X'(F')$ is obvious, it is enough to prove the opposite inclusion.
Let $P$ be an arbitrary $F'$-rational point on $X'$.
By the assumption, $f(P)$ is in $X(F)$.
Since $f$ is of degree 2, $P$ is in $X'(F'')$ for some extension $F''/F$ with $[F'':F]\leq  2$.
Because $[F':F]$ is odd, we have $F' \cap F''=F$ and thus 
$P\in X'(F')\cap X'(F'')=X'(F)$.
\end{proof}

\begin{cor}
\label{btos}
Let $F'/F$ be a finite extension of fields of characteristic 0 such that   $[F':F]$ is odd.
\begin{enumerate}
    \item If $X_0(15)(F')=X_0(15)(F)$, then $X(s3,b5)(F')=X(s3,b5)(F)$.
    \item If $X_0(21)(F')=X_0(21)(F)$, then $X(s3,b7)(F')=X(s3,b7)(F)$.
\end{enumerate}
\end{cor}
\begin{proof}
This is an immediate consequence of Lemma \ref{mc} (2), (4),  and Lemma \ref{points}.
\end{proof}

We also need the following results on the modularity of elliptic curves over totally real fields  with irreducible mod 5 or 7 representations.
\begin{thm}[{\cite[Theorem 7.6]{Th:irred}}]
\label{Th:irred}
Let $F$ be a totally real field which does not contain $\sqrt{5}$. Then every elliptic curve over $F$ whose mod $5$ representation is irreducible is modular. 
\end{thm}
\begin{thm}[{\cite[Theorem 1.5]{Y:ab}}]
\label{Y:irred}
Let $F$ be a totally real field in which $7$ is unramified.
Then every elliptic curve over $F$ whose mod $7$ representation is irreducible is modular.
\end{thm}
Applying these two modularity results together with Corollary \ref{btos}, we obtain the following corollary.
\begin{cor} 
\label{reduce57}
Let $F$ be a real quadratic field and $p$ be an odd prime number.
Let $F'$ be a totally real field contained in the cyclotomic $\Z_p$-extension $F_\infty$ of $F$.
Assume that at least one of the following conditions hold:
\begin{enumerate}
    \item $\sqrt{5}\notin F$ and  $X_0(15)(F_\infty)=X_0(15)(F)$.
    \item $7$ is unramified in $F_\infty$ and 
    $X_0(21)(F_\infty)=X_0(21)(F)$.
\end{enumerate}
Then all elliptic curves over $F'$ are modular.
\end{cor}
\begin{proof}
Let $E$ be an elliptic curve over $F'$.
Assume that the condition (1) holds. If $\overline{\rho}_{E,5}$ is irreducible, then $E$ is modular by Theorem \ref{Th:irred}.
If $\overline{\rho}_{E,5}$ is reducible, then $E$ gives an $F'$-rational point on $X(s3,b5)$ or $X_0(15)$. Then the condition (1) combined with Corollary \ref{btos} (1) implies that $P$ is indeed an $F$-rational point. Therefore, the $j$-invariant of $E$ is in $F$, and the modularity of $E$ follows from \cite[Theorem 1]{FLS} and the cyclic base change \cite{La}.

If we assume the condition (2), the same argument as above shows that $E$ is modular by considering $\overline{\rho}_{E,7}$ instead of $\overline{\rho}_{E,5}$ and using Theorem \ref{Y:irred} and Corollary \ref{btos} (2).
This completes the proof.
\end{proof}

\section{The proof of Theorem \ref{main5} and Theorem \ref{main7}}
\label{mainproof}
\subsection{The modular curves
\texorpdfstring{$X_0(15)$}{X0(15)} and
\texorpdfstring{$X_0(21)$}{X0(21)}
 and their twists}
In this section, we study $F_\infty$-rational points on $X_0(15)$ and $X_0(21)$, where $F_\infty$ is the cyclotomic $\Z_p$-extension of a real quadratic field $F$.

First, we list some properties and invariants of $X_0(15)$ and $X_0(21)$ that we shall need.
\begin{lem}
\label{mcproperties}
The following hold.
\begin{enumerate}
    \item For each prime $l\geq 3$, the mod $l$ representations $\overline{\rho}_{X_0(15),l}$ and  $\overline{\rho}_{X_0(21),l}$ are both surjective onto $\GL_2(\F_l)$.
    \item The Mordell--Weil groups $X_0(15)$ and $X_0(21)$ are both isomorphic to $\Z/2\Z\times \Z/4\Z$.
    \item Let $p$ be an odd prime number.
    For each odd prime number $l$, 
\[
X_0(15)(\Q_\infty)[l^\infty]=
X_0(15)(\Q)[l^\infty] =0, \text{ and}
\]
\[
X_0(21)(\Q_\infty)[l^\infty]=
X_0(21)(\Q)[l^\infty] =0.
\]
Also, 
\[
X_0(15)(\Q_\infty)[2^\infty] = X_0(15)(\Q)[2^\infty] = X_0(15)(\Q), \text{ and}
\]
\[
X_0(21)(\Q_\infty)[2^\infty] = X_0(21)(\Q)[2^\infty] = X_0(21)(\Q).
\]
\item The products of  Tamagawa numbers are given by \[\prod_{l=3,5}c_l(X_0(15))=\prod_{l=3,7}c_l(X_0(21))=8.\]
\item The central $L$-values divided by the real periods are given by 
    \[L(X_0(15),1)/\Omega_{X_0(15)}=L(X_0(21),1)/\Omega_{X_0(21)}=1/8.\]
\end{enumerate}
\end{lem}
\begin{proof}
Let $X$ denote $X_0(15)$ or $X_0(21)$.
LMFDB \cite{L} tells us that, for each $l\geq 3$, the mod $l$ representations attached to $X$ have the image containing $\SL_2(\F_l)$. Then the part (1) follows because the determinant of the mod $l$ representations $\overline{\rho}_{X,l}$ is the mod $l$  cyclotomic character of $\Gal(\overline{\Q}/\Q)$ and is surjective onto $\F_l^\times$.
The parts (2), (4), and (5) also follow from LMFDB \cite{L} or Cremona's table  \cite{Cr}.
Finally we shall check the part (3) below. 

Let us first consider the case $l\geq 3$.  Suppose contrarily that $X(\Q_\infty)[l^\infty]\neq 0$. Then in particular $X(\Q_\infty)[l]\neq 0$ and the image $\overline{\rho}_{X,l}(\Gal(\overline{\Q}/\Q_\infty))$ is contained in the group
\[
H=\left\{ 
\begin{pmatrix}
1 & b\\
0 & d
\end{pmatrix}\in \GL_2(\F_l)
\middle| \ b\in \F_l, d\in \F_l^\times \right\}\]
for some choice of a basis of $E[l](\overline{\Q})$.
On the other hand, by the part (1), we have  $\overline{\rho}_{X,l}(\Gal(\overline{\Q}/\Q))=\GL_2(\F_l)$.
Also, the index of $\overline{\rho}_{X,l}(\Gal(\overline{\Q}/\Q_\infty))$ in $\overline{\rho}_{X,l}(\Gal(\overline{\Q}/\Q))$ is a power of $p$.
Thus in particular the index of $H$ in $\GL_2(\F_l)$ is a power of $p$ as well; that is, $l^2-1$ is a power of $p$. This implies $p=2$, which  contradicts the assumption that $p$ is odd. (See also the proof in \cite{Z} for the corresponding statement.)

Next we consider the case $l=2$.
Let $P\in X(\Q_\infty)[2^\infty]$. We want to show that $P$ is $\Q$-rational. 
Since the coordinates of $P$ are in $\Q_\infty \cap \Q(X[2^\infty])$, it suffices to show that $\Q_\infty \cap \Q(X[2^\infty])=\Q$.  
By the part (2), we know that all the 2-torsion points on $X$ is $\Q$-rational, and so the  $2$-adic representation
\[
\rho_{X,2}\colon\Gal(\overline{\Q}/\Q)\to \GL_2(\Z_2)\]
 factors through
\begin{align*} \Gal(\Q(X[2^\infty])/\Q)&\hookrightarrow 
\ker(\GL_2(\Z_2)\to \GL_2(\Z/2\Z))\\
&=\begin{pmatrix}
1+2\Z_2 & 2\Z_2\\
2\Z_2 & 1+2\Z_2
\end{pmatrix}.
\end{align*}
Thus the extension $\Q(X[2^\infty])/\Q$ is a pro-$2$ extension. Since $p$ is odd, we have $\Q_\infty \cap \Q(X[2^\infty])=\Q$ as required.
\end{proof}

We deduce from Lemma \ref{mcproperties} a similar result for quadratic twists of $X_0(15)$ and $X_0(21)$.
\begin{lem}
\label{twistedmcproperties}
Let $d>1$ be a square-free integer greater. Let $X_0(15)_d$ (resp. $X_0(21)_d$) denote the quadratic twist of $X_0(15)$ (resp. $X_0(21)$) by the nontrivial Dirichlet character associated with the quadratic extension $\Q(\sqrt{d})/\Q$.
\begin{enumerate}
    \item For each prime $l\geq 3$, the mod $l$ representations $\overline{\rho}_{X_0(15)_d,l}$ and  $\overline{\rho}_{X_0(21)_d,l}$ are both surjective onto $\GL_2(\F_l)$.
    \item For an abelian group $A$, let $A^\mathrm{tors}$ denotes the torsion part of $A$. Then
    \[X_0(15)_d[2](\overline{\Q})\subset X_0(15)_d(\Q)^{\mathrm{tors}}\subset X_0(15)_d[2^\infty](\overline{\Q})\]
    and
    \[X_0(21)_d[2](\overline{\Q})\subset X_0(21)_d(\Q)^{\mathrm{tors}}\subset X_0(21)_d[2^\infty](\overline{\Q}).\]
     
    \item Let $p$ be an odd prime number. Then for each odd prime $l$, 
\[
X_0(15)_d(\Q_\infty)[l^\infty]=
X_0(15)_d(\Q)[l^\infty] =0, \text{ and}
\]
\[
X_0(21)_d(\Q_\infty)[l^\infty]=
X_0(21)_d(\Q)[l^\infty] =0.
\]
Also,
\[
X_0(15)_d(\Q_\infty)[2^\infty] = X_0(15)_d(\Q)[2^\infty] = X_0(15)_d(\Q), \text{ and}
\]
\[
X_0(21)_d(\Q_\infty)[2^\infty] = X_0(21)_d(\Q)[2^\infty] = X_0(21)_d(\Q).
\]
     \item The prime factors appearing in the Tamagawa numbers  $\prod_{l\mid 15d}c_l(X_0(15)_d)$ or $\prod_{l\mid 21d}c_l(X_0(21)_d)$ are $2$ or $3$.
\end{enumerate}
\end{lem}
\begin{proof}
The part (1) follows immediately from Lemma \ref{mcproperties} (1); in fact, the surjectivity of a mod $l$ representation is preserved by a twist by any (quadratic) character. 

By Lemma \ref{mcproperties} (2), the mod $2$ representations  $\overline{\rho}_{X_0(15),2}$ and $\overline{\rho}_{X_0(21),2}$ are both trivial, and so their twists by a quadratic character are also trivial. This shows that $X_0(15)_d(\Q)^{\mathrm{tors}}$  (resp. $X_0(21)_d(\Q)^{\mathrm{tors}}$) contains  $X_0(15)_d[2](\overline{\Q})$ (resp. $X_0(21)_d[2](\overline{\Q})$). The other inclusions are an immediate consequence of the part (1).

By using the parts (1) and (2), the same argument as Lemma \ref{mcproperties} (3) shows the part (3). 

The part (4) follows from Lemma \ref{mc} (5) and the description of special fibers of elliptic curves over local fields due to Kodaira and N\'eron; see \cite[Theorem 6.1]{Si} for example.
\end{proof}

The following property of quadratic twists of elliptic curves is well-known, but we record it with a proof for the readers' convenience.
\begin{lem}
\label{twist}
Let $E$ be an elliptic curve over a field $K$ of characteristic different from $2$.
Let $L=K(\sqrt{d})$  ($d\in K$) be a quadratic extension of $K$ and write  $\sigma\in \Gal(L/K)$ for the nontrivial element.
Then we have an isomorphism $E(L)^{\sigma=-1}\simeq E_d(K)$ of groups, where $E_d$ denotes the quadratic twist of $E$ by $d$.
\end{lem}
\begin{proof}
Take a Weierstrass equation defining $E$ of the form
\[y^2=x^3+ax^2+bx+c\quad (a,b,c\in K).\]
Then by definition of a quadratic twist, $E_d$ is given by the Weierstrass equation 
\[dy^2=x^3+ax^2+bx+c, \]
and the isomorphism $\varphi\colon E\stackrel{\simeq}{\longrightarrow}E_d$ over $L$ is  given by $\varphi(x,y)= (x, y/\sqrt{d})$.
One can check that $E(L)^{\sigma=-1}$ corresponds to $E_d(K)$ under this isomorphism.
\end{proof}
\begin{prop}
\label{15growth}
Let $p$ be a prime number such that $p\neq 2, 3$ or $5$.
\begin{enumerate}
    \item  $X_0(15)(\Q_\infty)=X_0(15)(\Q)$.
    \item Let $d>1$ be a square-free integer such that $(d,3p)=1$ or $(d,5p)=1$.
    If $L(X_0(15)_d,1)/\Omega_{X_0(15)_d}$ is a $p$-adic unit, then   $X_0(15)_d(\Q_\infty)=X_0(15)_d(\Q)$.
\end{enumerate}
\end{prop}
\begin{prop}
\label{21growth}
Let $p$ be a prime number not dividing such that $p\neq 2, 3$ or $7$.
\begin{enumerate}
    \item $X_0(21)(\Q_\infty)=X_0(21)(\Q)$.
    \item Let $d>1$ be a square-free integer such that $(d,3p)=1$ or $(d,7p)=1$.
    If $L(X_0(21)_d,1)/\Omega_{X_0(21)_d}$ is a $p$-adic unit, then   $X_0(21)_d(\Q_\infty)=X_0(21)_d(\Q)$.
\end{enumerate}
\end{prop}
Our proof of Proposition \ref{15growth} and Proposition \ref{21growth} are exactly the same.
Also, Proposition \ref{15growth} (1) is included in \cite{Th:cyclo}.
Thus we only give a proof of Proposition \ref{21growth}.
\begin{proof}[Proof of Proposition \ref{21growth}]
The strategy of the proof is similar to \cite{Th:cyclo}.
Let 
$Y$ denote $X_0(21)$ or $X_0(21)_d$.
To prove the proposition, we need to show that $Y(\Q_\infty)$ and $Y(\Q)$ have the same torsion part and that $Y(\Q_\infty)$ is of rank $0$.
The claim about the torsion part follows immediately from Lemma \ref{mcproperties} (3) and Lemma \ref{twistedmcproperties} (3). So it suffices to show that $Y(\Q_\infty)$ is finite.
Note that $Y$ has good reduction at $p$ because of the choice of $p$ and $d$.

Assume first that $Y$ has good ordinary reduction at $p$. We shall check that $Y$ satisfies the conditions (1)-(6) in Corollary \ref{ordfin}.
\begin{itemize}
    \item The condition (1) holds since we  assume here that $Y$ has good ordinary reduction at $p$.
    \item The condition (2) holds by Lemma \ref{mcproperties} (1) or Lemma \ref{twistedmcproperties} (1).
    \item We check the condition (3). By Lemma \ref{mc} (5) and \cite[Proposition 2.12 (c)]{DDT}, we see that  $\overline{\rho}_{Y,p}$ is ramified at $3$ and $7$. Thus, to check  the condition (3), it is enough to show that $3||N_Y$ or $7||N_Y$. This is obvious when $Y=X_0(21)$, since $X_0(21)$ is the elliptic curve with Cremona label 21A1, which has the conductor $21$. Next, consider the case where $Y=X_0(21)_d$. By the assumption that  $(d,3)=1$ or $(d,7)=1$, we see that $Y$ has multiplicative reduction at $3$ or $7$, and hence we have $3||N_Y$ or $7||N_Y$, where $N_Y$ is the conductor of $Y$.
    \item If $Y=X_0(21)$, then the condition (4) follows from Lemma \ref{mcproperties} (4). If $Y=X_0(21)_d$, then it follows  from the assumption. 
    \item We check the condition (5) as follows.
    Suppose contrarily that $p$ divides $\# Y(\F_p)$.  In both cases where $Y=X_0(21)$ and $X_0(21)_d$, we have $4p\mid \# Y(\F_p)$; in fact, $Y[2](\Q)$ injects into $Y(\F_p)$, and $Y(\Q)$ contains all the 2-torsion points of $Y$ by Lemma \ref{mcproperties} (2) and Lemma \ref{twistedmcproperties} (2).
    However, this contradicts to the Hasse bound $|1+p-\#Y(\F_p)|\leq 2\sqrt{p}$. 
    \item The condition (6) follows from Lemma \ref{mcproperties} (2) and Lemma \ref{twistedmcproperties} (2).
\end{itemize}
Therefore, we can apply Corollary \ref{ordfin} to $Y$ and thus $Y(\Q_\infty)$ is finite.
    
     Assume next that $Y$ has good supersingular reduction at $p$. We check that $Y$ satisfies the conditions in Theorem \ref{Ku}.
     \begin{itemize}
    \item The condition (1) holds since we  assume here that $Y$ is ordinary at $p$.
    \item The condition (2) holds by Lemma \ref{mcproperties} (1) and Lemma \ref{twistedmcproperties} (1).
    \item If $Y=X_0(21)$, then the condition (3) follows from Lemma \ref{mcproperties} (5). If $Y=X_0(21)_d$, then it follows  from the assumption.
    \item The condition (4) for $Y=X_0(21)$ follows from Lemma \ref{mcproperties} (4).  For $Y=X_0(3q)_d$, it follows from Lemma \ref{twistedmcproperties} (4).
     \end{itemize}
     Hence we can apply Theorem \ref{Ku} to $Y$ and obtain the finiteness of $Y(\Q_\infty)$.
\end{proof}
The following are the key to our proof of Theorem \ref{main5} and Theorem \ref{main7}.
\begin{prop}
\label{keyprop15}
Let $p$ be a prime number such that $p\neq 2, 3$ or $5$, and $d>1$ be a square-free integer such that $(d,3p)=1$ or $(d,5p)=1$.
Assume that $L(X_0(15)_d,1)/\Omega_{X_0(15)_d}$ is a $p$-adic unit.
Then we have $X_0(15)(\Q(\sqrt{d})_\infty)=X_0(15)(\Q(\sqrt{d}))$.
\end{prop}

\begin{prop}
\label{keyprop21}
Let $p$ be a prime number such that $p\neq 2, 3$ or $7$, and $d>1$ be a square-free integer such that $(d,3p)=1$ or $(d,7p)=1$.
Assume that $L(X_0(15)_d,1)/\Omega_{X_0(15)_d}$ is a $p$-adic unit.
Then we have $X_0(21)(\Q(\sqrt{d})_\infty)=X_0(21)(\Q(\sqrt{d}))$.
\end{prop}
The proof of these propositions are exactly the same, and so we only prove Proposition \ref{keyprop21}.
\begin{proof}[Proof of Proposition \ref{keyprop21}]
Write $X=X_0(21)$ and $F=\Q(\sqrt{d})$ to simplify the notation. Also, let $\sigma$ denote the nontrivial element of $\Gal(F_\infty/\Q_\infty)\cong \Gal(F/\Q)$.

We claim that $X(F_\infty)\subset X[2^\infty](\overline{F})$.
Note that
\[2X(F_\infty)\subset X(F_\infty)^{\sigma=\mathrm{id}}+X(F_\infty)^{\sigma=-\mathrm{id}}\]
because for any $P\in X(F_\infty)$ we have $2P=(P+\sigma P)+(P-\sigma P)$.
By Proposition \ref{21growth} (1) and Lemma \ref{mcproperties} (2), we have 
\[X(F_\infty)^{\sigma=\mathrm{id}}=X(\Q_\infty)
=X(\Q) 
\simeq \Z/2\Z\times \Z/4\Z.\]
Also, by Lemma \ref{twist}, Proposition \ref{21growth} (2), and Lemma \ref{twistedmcproperties} (2), we have
\[X(F_\infty)^{\sigma=-\mathrm{id}}\simeq X_d(\Q_\infty)=X_d(\Q)\subset X_d[2^\infty](\overline{\Q}).\]
This shows the claim.

Let $P\in X(F_\infty)$. By the above claim, $P$ is also in  $X[2^\infty](\overline{F})$.
Thus the coordinates of $P$ belongs to $F_\infty\cap F(X[2^\infty])$. By the same argument as in the proof of Lemma \ref{mcproperties} (3) for $l=2$, we obtain $F_\infty\cap F(X[2^\infty])=F$ and hence we have $P\in X(F)$.
\end{proof}
\subsection{Completing the proof of Theorem \ref{main5} and Theorem \ref{main7}}

Theorem \ref{main5} follows from  Corollary \ref{reduce57} (1) and Proposition \ref{keyprop15}.
Also, Theorem \ref{main7} follows from Corollary \ref{reduce57} (2) and Proposition \ref{main7}.
\section{Examples}
\label{examples}
In this section, we see some examples where Theorem \ref{main5} or Theorem \ref{main7} can be applied.
First we give a description of the conductor of quadratic twists of a semi-stable elliptic curve over $\Q$.


Let $E$ be an elliptic curve over $\Q$ of square-free conductor $N_E$.
Let $d$ be a square-free integer
and $D$ be the discriminant of the quadratic field $\Q(\sqrt{d})$; that is, $D=d$ if $d\equiv 1$ mod $4$ and $D=4d$ otherwise.
For each prime $p$, let $c_p$ be the $p$-part of the conductor of the quadratic extension $\Q(\sqrt{d})/\Q$; it is given by
\[c_p=
\begin{cases}
1 & \text{if } p\mid d\\
0 & \text{if } p\nmid d
\end{cases}
\]
if $p$ is odd, and 
\[c_2=
\begin{cases}
0 & \text{if } d\equiv 1 \text{ mod } 4\\
3 & \text{if }  d\equiv 2 \text{ mod } 4\\
2 & \text{if }  d\equiv 3 \text{ mod } 4.
\end{cases}
\]
If we write $N_{E_d}=\prod_{p}p^{e_p}$ for the conductor of the quadratic twist $E_d$, then the exponent $e_p$ of $p$ is given as follows:
\begin{itemize}
    \item If $p\nmid N_E$, then $e_p=2c_p$.
   \item Otherwise, $e_p=2c_p$; that is, $e_p=2c_p$ if $p\mid D$, and $e_p=0$ if $p\nmid N_E D$.
\end{itemize}

Using this description, we can compute the conductor of the quadratic twists $X_0(15)_d$ and $X_0(21)_d$.
Examining LMFDB \cite{L}, we then find the values $L(X_0(15)_d,1)/\Omega_{X_0(15)_d}$ and $L(X_0(21)_d,1)/\Omega_{X_0(21)_d}$.

Table 1 below is a list of $L(X_0(15)_d,1)/\Omega_{X_0(15)_d}$ for several $d$ as considered in Theorem \ref{main5} (or in Remark \ref{equiv} (1)); that is, $d>1$ is a square-free integer such that  $d\neq 5$ and that $(d,3)=1$ or $(d,5)=1$. The fourth column  in Table 1 lists prime numbers $p$ such that Theorem \ref{main5} can be applied for the given $d$. When the value $L(X_0(15)_d,1)/\Omega_{X_0(15)_d}$ vanishes, the assumption of Theorem \ref{main5} is not satisfied for any choice of $p$; in this case, we write ``none'' at the entry of the table.

Similarly, Table 2 below shows the values of $L(X_0(21)_d,1)/\Omega_{X_0(21)_d}$ for several $d$ and prime numbers $p$ for which Theorem \ref{main7} can be applied for the given $d$. Note that we list $d$'s which satisfy the conditions of Theorem \ref{main7} (or Remark \ref{equiv} (2)); that is, $d>1$ is a square-free integer with $7\nmid d$.
\begin{table}[h]
  \begin{center}
    \caption{quadratic twists $X_0(15)_d$} 
	\label{15twists} 
    \begin{tabular}{ccccc}\hline 
      $d$  & $N_{X_0(15)_d}$   & \begin{tabular}{c}Cremona label of\\ $X_0(15)_d$\end{tabular} & $L(X_0(15)_d,1)/\Omega_{X_0(15)_d}$ & $p$ \\
      \hline
      $2$ & $2^6 \cdot 3 \cdot 5$ & 960g3 & $2$ & $p\neq 2,3,5$ \\
      $3$ & $2^4\cdot 3^2\cdot 5$ & 720h3 & $0$ & none \\
      $6$ & $2^6\cdot 3^2\cdot 5$ & 2880bd3 & $4$ & $p\neq 2,3,5$ \\
      $7$ & $2^4\cdot 3\cdot 5\cdot 7^2$ & 11760bq3 & $0$ & none\\
      $10$ & $2^6\cdot 3\cdot 5^2$ & 4800b3 & $0$ & none \\
      $11$ & $2^4\cdot 3\cdot 5\cdot 11^2$ & 29040dg3 & $0$ & none \\
      $13$ & $3\cdot 5\cdot 13^2$ & 2535a3 & $0$ & none\\
      $14$ & $2^6\cdot 3\cdot 5\cdot 7^2$ & 47040hg3 & $0$ & none\\
      $17$ & $3\cdot 5\cdot 17^2$ & 4335d3 & $2$ & $p\neq 2,3,5,17$\\
      $19$ & $2^4\cdot 3\cdot 5\cdot 19^2$ & 86640cm3 & $8$ & $p\neq 2,3,5,19$\\
      $21$ & $3^2\cdot 5\cdot 7^2$ & 2205j3 & $ 4 $ & $p\neq 2,3,5,7$ \\
      $22$ & $2^6\cdot 3\cdot 5\cdot 11^2$ & 116160ez3 & $0$ & none\\
      $23$ & $2^4\cdot 3\cdot 5\cdot 23^2$ & 126960cj3 & $8$ & $p\neq 2,3,5,23$\\
      $26$ & $2^6\cdot 3\cdot 5\cdot 13^2$ & 162240ez4 & $0$ & none\\
      $29$ & $3\cdot 5\cdot 29^2$ & 12615f3 & $0$ & none \\
      $31$ & $2^4\cdot 3\cdot 5\cdot 31^2$ & 230640bg4 & $8$ &  $p\neq 2,3,5,31$\\
      $33$ & $3^2\cdot 5\cdot 11^2$ & 5445g3 & $0$ & none \\
      $34$ & $2^6\cdot 3\cdot 5\cdot 17^2$ & 277440do4 & $0$ & none\\
      $35$ & $2^4\cdot 3\cdot 5^2\cdot 7^2$ & 58800it3 & $16$ &  $p\neq 2,3,5,7$\\
      $37$ & $3\cdot 5\cdot 37^2$ & 20535a3 & $0$ & none \\
      $38$ & $2^6\cdot 3\cdot 5\cdot 19^2$ & 346560gv4 & $0$ &  none\\
      $39$ & $2^4\cdot 3^2\cdot 5\cdot 13^2$ & 121680en3 & $16$ & $p\neq 2,3,5,13$ \\
      $41$ & $3\cdot 5\cdot 41^2$ & 25215h3 & $0$ & none \\
      \hline
    \end{tabular}
  \end{center}
\end{table}
\newpage
\begin{table}[h]
  \begin{center}
    \caption{quadratic twists $X_0(21)_d$} 
	\label{21twists} 
    \begin{tabular}{ccccc}\hline 
      $d$   & $N_{X_0(21)_d}$  & \begin{tabular}{c}Cremona label of\\ $X_0(21)_d$ \end{tabular}& $L(X_0(21)_d,1)/\Omega_{X_0(21)_d}$ & $p$\\
      \hline
      $2$ & $2^6 \cdot 3 \cdot 7$ & 1344a2 & $0$ & none\\
      $3$ & $2^4\cdot 3^2\cdot 7$& 1008k2 & $2$ & $p\neq 2,3,7$\\
      $5$ & $3\cdot 5^2\cdot 7$ & 525b2 & $1$ & $p\neq  2,3,5,7$\\
      $6$ & $2^6\cdot 3^2\cdot 7$ & 4032bm2 & $2$ & $p\neq 2,3,7$\\
      $10$ & $2^6\cdot 3\cdot 5^2\cdot 7$ & 33600dd2 & $0$ & none\\
      $11$ & $2^4\cdot 3\cdot 7\cdot 11^2$ & 40656bk2 & $0$ & none \\
      $13$ & $3\cdot 7\cdot 13^2$ & 3549c2 & $0$ & none\\
      $15$ & $2^4\cdot 3^2\cdot 5^2\cdot 7$ & 25200dx2 & $0$ & none\\
      $17$ & $3\cdot 7\cdot 17^2$ & 6069b2 & $1$ & $p\neq 2,3,7,17$\\
      $19$ & $2^4\cdot 3\cdot 7\cdot 19^2$ & 121296dk2 & $0$ & none \\
      $22$ & $2^6\cdot 3\cdot 7\cdot 11^2$ & 162624bj2 & $8$ & $p\neq 2,3,7,11$\\
      $23$ & $2^4\cdot 3\cdot 7\cdot 23^2$ & 177744ca2 &  $0$ & none\\
      $26$ & $2^6\cdot 3\cdot 7\cdot 13^2$ & 227136ho2  & $4$ & $p\neq 2,3,7,13$\\
      $29$ & $3\cdot 7\cdot 29^2$ & 17661a2 & $0$ & none\\
      $30$ & $2^6\cdot 3^2\cdot 5^2\cdot 7$  & 100800me2 & $0$ & none\\
      $31$ & $2^4\cdot 3\cdot 7\cdot 31^2$ & 322896cn2 & $0$ & none\\
      $33$ & $3^2\cdot 7\cdot 11^2$ & 7623p2 & $2$ & $p\neq 2,3,7,11$\\
      $34$ & $2^6\cdot 3\cdot 7\cdot 17^2$ & 388416fo2 & $0$ & none\\
      $37$ & $3\cdot 7\cdot 37^2$ & 28749e2 & $8$ & $p\neq 2,3,7,37$\\
      $38$ & $2^6\cdot 3\cdot 7\cdot 19^2$ & 485184dx2 & $4$ & $p\neq 2,3,7,19$\\
      $39$ & $2^4\cdot 3^2\cdot 7\cdot 13^2$  & 170352t2 & $0$ & none\\
      $41$ & $2^4\cdot 3\cdot 7\cdot 43^2$ & 35301e2 & $1$ & $p\neq 2,3,7,41$\\
      \hline
    \end{tabular}
  \end{center}
\end{table}

From Table 1 and Table 2, we observe the following:
\begin{itemize}
    \item Note that $d=7,14,21$, and $35$ appear only in Table 1. Among them, for $d=21$ and $35$ the value $(X_0(15)_d,1)/\Omega_{X_0(15)_d}$ do not vanish, and hence Theorem \ref{main5} can be applied in these cases.
    \item Note that $d=5,15$, and $30$ appear only in Table 2.
    The value $(X_0(15)_d,1)/\Omega_{X_0(15)_d}$ for $d=5$ does not vanish, and hence Theorem \ref{main7} can be applied (for $p=2$ or $p\geq 11$.)
    \item For $d=10,11,13,29$, and $34$, the central $L$-values vanish in both tables vanish, and hence we cannot use Theorem \ref{main5} or Theorem \ref{main7}  for such $d$'s.
    \item For $d=2,3,6,17,19,22,23,26,31,33,37,38,39$,  and $41$,  $L(X_0(15)_d,1)/\Omega_{X_0(15)_d}$ or
    $L(X_0(21)_d,1)/\Omega_{X_0(21)_d}$ do not vanish, and hence Theorem \ref{main5} or Theorem \ref{main7} can be applied in these cases (for almost all primes $p$).
\end{itemize}

\subsection*{Acknowledgement}
The author would like to express his sincere gratitude to Tetsushi Ito for  his continuous encouragement and invaluable advice.  This work was also supported by JSPS KAKENHI Grant Number 19K14514.

\bibliographystyle{amsplain}

\begin{thebibliography}{99}


\bibitem{Bo}
  Box, J.,
  \textit{Elliptic curves over totally real quartic fields not containing $\sqrt{5}$ are modular},
  Trans.\ Amer.\ Math.\ Soc.\ 375 (2022), no.\ 5, 3129-3172.

\bibitem{BCDT}
  Breuil, C., Conrad, B., Diamond, F., Taylor, R.,
  \textit{On the modularity of elliptic curves over $\Q$: wild $3$-adic exercises},
  J.\ Amer.\ Math.\ Soc.\ 14 (2001), no.\ 4, 843-939.

\bibitem{Cr}
  Cremona, J.\ E.,
  \textit{Algorithms for modular elliptic curves},
  Second edition.\ Cambridge University Press, Cambridge, 1997.

\bibitem{DDT}
   Darmon, H., Diamond, F., Taylor, Richard., \textit{Fermat's last theorem. Elliptic curves, modular forms \& Fermat's last theorem (Hong Kong, 1993)}, 2–140, Int. Press, Cambridge, MA, 1997. 
  
  
\bibitem{DNS}
  Derickx, M., Najman, F., Siksek, S.,
  \textit{Elliptic curves over totally real cubic fields are modular},
  Algebra Number Theory \textbf{14} (2020), no.\ 7, 1791-1800.

\bibitem{FLS}
  Freitas, N., Le Hung, B.\ V., Siksek, S.,
  \textit{Elliptic curves over real quadratic fields are modular},
  Invent.\ Math.\ 201 (2015), no.\ 1, 159-206.

\bibitem{GZN}
  Gonz\'alez-Jim\'enez, E., Najman, F.,
  \textit{Growth of torsion groups of elliptic curves upon base change},
  Math.\ Comp.\ 89 (2020), no.\ 323, 1457-1485.

\bibitem{Gr}
  Greenberg, R.,
  \textit{Iwasawa theory for elliptic curves},
  Arithmetic theory of elliptic curves (Cetraro, 1997), 51-144, Lecture Notes in Math., 1716, Springer, Berlin, 1999.

\bibitem{Im}
  Imai, H.,
  \textit{A remark on the rational points of abelian varieties with values in cyclotomic Zp-extensions}, Proc. Japan Acad. 51 (1975), 12–16. 

\bibitem{IIY}
  Ishitsuka, Y., Ito, T., Yoshikawa, S.,
  \textit{The modularity of elliptic curves over all but finitely many totally real fields of degree 5},
  preprint, 2021. 
  (\url{https://arxiv.org/abs/2110.04078})

\bibitem{Ka}
   Kato, K.,
   \textit{$p$-adic Hodge theory and values of zeta functions of modular forms},
   Cohomologies $p$-adiques et applications arithm\'etiques.\ III. Ast\'erisque No.\ 295 (2004), ix, 117-290.

\bibitem{Ku}
  Kurihara, M.,
  \textit{On the Tate Shafarevich groups over cyclotomic fields of an elliptic curve with supersingular reduction.\ I},
  Invent.\ Math.\ 149 (2002), no.\ 1, 195-224.

\bibitem{La}
 Langlands, Robert P.,  \textit{Base change for GL(2)}, Annals of Mathematics Studies, No. 96 Princeton University Press, Princeton, N.J.; University of Tokyo Press, Tokyo, 1980. vii+237 pp. 

\bibitem{L}
  The LMFDB Collaboration, \textit{The $L$-functions and modular forms database}, 2022. \\
  (\url{https://www.lmfdb.org/})
  
\bibitem{Li}
Li, Wen Ch'ing Winnie.,
\textit{Newforms and functional equations},
Math. Ann. 212 (1975), 285–315.
  
\bibitem{Si}
  Silverman, Joseph H. \textit{The arithmetic of elliptic curves}, Vol. 106. New York: Springer, 2009.

\bibitem{SU}
  Skinner, C., Urban, E., \textit{The Iwasawa main conjectures for $\mathrm{GL}_2$}, Invent.\ Math.\ 195 (2014), no.\ 1, 1-277.

\bibitem{TW}
  Taylor, R., Wiles, A.,
  \textit{Ring-theoretic properties of certain Hecke algebras},
  Ann.\ of Math.\ (2) 141 (1995), no.\ 3, 553-572.

\bibitem{Th:irred}
 Thorne, J.\ A., \textit{Automorphy of some residually dihedral Galois representations}, Math.\ Ann.\ 364 (2016), no.\ 1-2, 589-648.

\bibitem{Th:cyclo}
  Thorne, J.\ A.,
  \textit{Elliptic curves over $\mathbb{Q}_{\infty}$ are modular},
  J.\ Eur.\ Math.\ Soc.\ (JEMS) 21 (2019), no.\ 7, 1943-1948.


\bibitem{Wi}
Wiles, A., \textit{Modular elliptic curves and Fermat's last theorem}, Ann.\ of Math.\ (2) 141 (1995), no.\ 3, 443-551.

\bibitem{Y:comp}
  Yoshikawa, S.,
  \textit{On the modularity of elliptic curves over a composite field of some real quadratic fields},
  Res.\ Number Theory 2 (2016), Paper No.\ 31, 6 pp.

\bibitem{Y:ab}
  Yoshikawa, S.,
  \textit{Modularity of elliptic curves abelian totally real fields unramified at 3, 5, and 7},
  J.\ Th\'eor.\ Nombres Bordeaux 30 (2018), no.\ 3, 729-741.

\bibitem{Z}
   Zhang, X.,
   \textit{On the modularity of elliptic curves over the cyclotomic $\mathbb{Z}_p$-extension of some real quadratic fields},
   preprint, 2022.
   (\url{https://arxiv.org/abs/2205.09790})
\end{thebibliography}

\end{document}